\documentclass[12pt]{article}
\usepackage{tikz-cd}
\usepackage{amsfonts}
\usepackage{mathrsfs}
\usepackage{amsxtra,amsthm,amsmath,amssymb, amscd, authblk}
\usepackage{stmaryrd}

\theoremstyle{plain}

\newtheorem{thm}{Theorem}[section]
\newtheorem{pro}[thm]{Proposition}

\newtheorem{lem}[thm]{Lemma}

\newtheorem{cor}[thm]{Corollary}

\theoremstyle{definition} 
\newtheorem{exa}[thm]{Example}
\newtheorem{exas}[thm]{Examples}
\newtheorem{defn}[thm]{Definition}
\newtheorem{rem}[thm]{Remark}

\newtheorem{lem-defn}[thm]{Lemma-Definition}

\renewcommand{\qed}{\begin{flushright} {\bf Q.E.D.}\ \ \ \ \
                  \end{flushright} }

\newcommand{\hs}{\hspace{.2in}}

\newcommand{\IC}{\mathbb{C}}

\newcommand{\IK}{\mathbb{K}}
\newcommand{\IN}{\mathbb{N}}

\newcommand{\IR}{\mathbb{R}}

\newcommand{\XX}{\mathfrak{X}}

\def \g  {\mathfrak{a}}
\def \b  {\mathfrak{b}}

\def \g  {\mathfrak{g}}

\def \i  {\mathfrak{i}}

\def \n  {\mathfrak{n}}

\def \p  {\mathfrak{p}}

\def \u  {\mathfrak{u}}

\def \t  {\mathfrak{t}}

\def \ulamb  {\underline{\lambda}}
\def \umu  {\underline{\mu}}

\def \C  {{\mathcal{C}}}

\def \H  {{\mathcal{H}}}

\def \O  {{\mathcal{O}}}
\def \P  {{\mathcal{P}}}

\def \R  {{\mathcal{R}}}

\def \U  {{\mathcal{U}}}

\DeclareMathOperator{\Ch}{Ch}

\DeclareMathOperator{\Der}{Der}

\DeclareMathOperator{\diag}{diag}

\DeclareMathOperator{\Mix}{Mix}
\DeclareMathOperator{\Twi}{Twi}

\def \sA {{\scriptscriptstyle A}}

\def \sH {{\scriptscriptstyle H}}
\def \sJ {{\scriptscriptstyle J}}

\def \sU {{\scriptscriptstyle U}}
\def \sV {{\scriptscriptstyle V}}
\def \sW {{\scriptscriptstyle W}}
\def \sX {{\scriptscriptstyle X}}

\def \sR {{\scriptscriptstyle R}}

\newcommand{\la}{\langle}
\newcommand{\ra}{\rangle}
\newcommand{\lara}{\la \, , \, \ra}

\newcommand{\kh}{\IK[[\hslash]]}
\newcommand{\ch}{\IC[[\hslash]]}

\newcommand{\Rep}{\R ep}

\newcommand{\pist}{\pi_{\rm st}}

\newcommand{\piG}{{\pi_{\scriptscriptstyle G}}}

\begin{document}

\setlength{\baselineskip}{1.2\baselineskip}

\title{Quantization of a Poisson structure on products of principal affine spaces}
\author[$\star$]{Victor Mouquin}
\affil[$\star$]{
School of Mathematical Sciences, Shanghai Jiaotong University, Shanghai, China, mouquinv@sjtu.edu.cn}
\date{}

\maketitle


\begin{abstract}
We give the analogue for Hopf algebras of the polyuble Lie bialgebra construction by Fock and Rosli. By applying this construction to the Drinfeld-Jimbo quantum group, we obtain a deformation quantization $\IC_\hslash[(N \backslash G)^m]$ of a Poisson structure $\pi^{(m)}$ on products $(N \backslash G)^m$ of principal affine spaces of a connected and simply connected complex semisimple Lie group $G$. The Poisson structure $\pi^{(m)}$ descends to a Poisson structure $\pi_m$ on products $(B \backslash G)^m$ of the flag variety of $G$ which was introduced and studied by the Lu and the author. Any ample line bundle on $(B \backslash G)^m$ inherits a natural flat Poisson connection, and the corresponding graded Poisson algebra is quantized to a subalgebra of $\IC_\hslash[(N \backslash G)^m]$. 

We define the notion of a strongly coisotropic subalgebra in a Hopf algebra, and explain how strong coisotropicity guarantees that any homogeneous coordinate ring of a homogeneous space of a Poisson Lie group can be quantized in the sense of Ciccoli, Fioresi, and Gavarini. 
\end{abstract}

\section{Introduction} \label{sec-intro}

Let $N \backslash G$ be the principal affine space of a connected and simply connected complex semisimple Lie group $G$, where $N$ is the unipotent radical of a Borel subgroup $B$. For each integer $m \geq 1$, we introduce in this paper a Poisson structure $\pi^{(m)}$ on $(N\backslash G)^m$ with the following properties: the diagonal action of $G$ on $((N \backslash G)^m, \pi^{(m)})$ is Poisson when $G$ is equipped with $\pi_{st}^{(1)}$, the standard multiplicative Poisson structure; $\pi^{(m)}$ is the quotient structure of the $m$'th polyuble Poisson Lie group $(G^m, \pi_{st}^{(m)})$, a term explained in $\S$\ref{subsec-poiss-products}; and $\pi^{(m)}$ descends to a Poisson structure $\pi_m$ on $(B \backslash G)^m$ which was extensively studied in \cite{Elek-Lu:BS, Lu-Mou:double-B-cell, Lu-Mou:mixed}, and as a consequence, any ample line bundle on $((B \backslash G)^m, \pi_m)$ admits a flat Poisson connection. 

The geometry of $(N\backslash G)^m$ is intimately related to the representation theory of $G$, or its Langland dual $G^L$. In \cite{gonch-shen}, Goncharov and Shen construct parametrization sets of canonical bases in invariants of tensor products of representations of $G^L$ from the tropical geometry of $(N\backslash G)^m/G$. An interesting problem is to relate the Poisson geometry developed in this paper to their work, as well as to the cluster structure introduced by Fock and Goncharov \cite{FG:moduli-space}.

Polyubles of quasitriangular Lie bialgebras, which are generalizations of doubles of Lie bialgebras, were first introduced by Fock-Rosli \cite{FR2} and further studied by J.-H. Lu and the author \cite{Lu-Mou:mixed}. Given a quasitriangular Hopf algebra $(H, R)$, we give in this paper the Hopf algebra analogue $(H_\sR^{(m)}, R^{(m)})$ of this construction. We also construct $(H_\sR^{(m)}, R^{(m)})$-module algebras which are locally factored in the sense of Etingof and Kazhdan  \cite{quant-li-bialg-III}, and are quantum analogues of the mixed product Poisson structures studied in \cite{Lu-Mou:mixed}. 

Applying this general theory to the Drinfeld-Jimbo QUE algebra $U_\hslash(\g)$, we obtain a Hopf algebra $\H^{(m)}$ which is a quantization of the Poisson algebra $(\IC[G^m], \{ \:, \:\}_{r_{st}}^{(m)})$ of regular functions on $(G^m, \pi_{st}^{(m)})$. Moreover, we identify a subalgebra $\IC_\hslash[(N \backslash G)^m]$ of $\H^{(m)}$ which is graded by $(\P_+)^m$, where $\P_+$ is the monoid of dominant weights of the maximal torus $T$ of $G$, and which is quantization of the Poisson algebra of regular functions on $((N \backslash G)^m, \pi^{(m)})$. Flat Poisson connections on ample line bundles over $((B \backslash G)^m, \pi_m)$ induce a graded Poisson bracket on the corresponding graded algebras, and in particular, any of those graded Poisson algebras are quantized to graded subalgebras of $\IC_\hslash[(N \backslash G)^m] $. We thus recover a particular example of a result from \cite{3italians} (albeit over $\ch$ rather than $\IC[q, q^{-1}]$) quantizing homogenous coordinate rings of homogeneous spaces of Poisson Lie groups. 

In \cite{3italians}, Ciccoli, Fioresi, and Gavarini construct quantizations of homogeneous coordinate rings using prequantum sections satisfying a certain involutivity condition, see $\S$\ref{subsec-3italians} and  \cite[Section 3]{3italians}. We introduce the definition of a (right or left) {\it strongly coisotropic subalgebra} of a Hopf algebra, and show that strong coisotropicity implies that any homogeneous coordinate ring of a homogeneous space of a Poisson Lie group can be quantized in the sense of \cite{3italians}. 

This paper is organized as follows: $\S$\ref{sec-basic} is a recall of the notion of Lie bialgebras and Poisson action thereof. In $\S$\ref{sec-prod-G/N} we construct the Poisson structure $\pi^{(m)}$ on $(N \backslash G)^m$ as an application of the theory in \cite{Lu-Mou:mixed}, and show that when $\IC[(N \backslash G)^m]$ is equipped with the induced Poisson bracket $\{ \:, \:\}^{(m)}$, it is a Poisson subalgebra of $(\IC[G^m], \{ \:, \:\}_{r_{st}}^{(m)})$. In $\S$\ref{sec-strong-coiso-liebialg} we define strongly coisotropic subalgebras of Lie bialgebras, a notion which generalizes to the graded Poisson algebra setting that of coisotropic subalgebras. In particular, with $\b$ being the Lie algebra of $B$, $\b^n$ is a strongly coisotropic subalgebra of the Lie bialgebra $(\g^m, r_{st}^{(m)})$ of $(G^m, \pi_{st}^{(m)})$, which implies that the Poisson bracket $\{ \:, \:\}^{(m)}$ on $\IC[(N \backslash G)^m]$ is graded by $(\P_+)^m$.

Section $\S$\ref{sec-hopf} is a recall on Hopf algebras, and in $\S$\ref{sec-twisted}, we develop in the Hopf algebra setting the analog of the  twisted $m$-fold product of a quasitriangular Lie bialgebra $(\g, r)$ introduced in \cite{Lu-Mou:mixed}. The quantum analogue of a mixed product Poisson structure with a Poisson action of $(\g, r)$ is a module algebra over a QUE algebra which is locally factored, and in $\S$\ref{sec-quantum-NbackslashG} the algebra $\IC_\hslash[(N \backslash G)^m]$ is obtained as such an example. We define in $\S$\ref{sec-strong-coiso-hopf} strongly coisotropic subalgebras of Hopf algebras which are analogues of strongly coisotropic subalgebras of Lie bialgebras, and explain how strong coisotropicity guarantees that any homogeneous coordinate ring of a homogeneous space of a Poisson Lie group can be quantized in the sense of \cite{3italians}.

\subsection{Notation} \label{subsec-nota}

The canonical pairing between a vector space and its dual is denoted by $\lara$. 

All algebras in this paper are finitely generated unital associative algebras over a field of zero characteristic. If $A$ is an algebra, we denote its unit element by $1_\sA$, its multiplication map by $\mu_\sA:A \otimes A \to A$, and when no confusion is possible, multiplication is also written as concatenation, $\mu_\sA(a_1 \otimes a_2) = a_1a_2$. If $R \in A^{\otimes l}$, $l \geq 1$, we use the standard notation $R_{i_1 i_2 \cdots i_l}$ to denote the image of $R$ in $A^{\otimes k}$, $k \geq l$, under the embedding $A^{\otimes l} \hookrightarrow A^{\otimes k}$ as the $(i_1, i_2, \ldots, i_l)$ components, $1 \leq i_1, i_2, \ldots, i_l \leq k$. 

If $G$ is a Lie group with Lie algebra $\g$ and $x \in \g^{\otimes l}$, $l \geq 1$, we denote respectively by $x^L$ and $x^R$ the left and right invariant tensor fields whose value at the identity $e \in G$ is $x$. 
 
\subsection*{Acknowledgements}

The author would like to thank Jiang-Hua Lu for her valuable suggestions and stimulating discussions, as well as Milen Yakimov and Nicola Ciccoli, for answering his numerous questions. The author is also thankful to the anonymous referee, for his many suggestions which have improved the paper.


\section{Lie bialgebras} \label{sec-basic}

We recall basic facts concerning Lie bialgebras and quasitriangular $r$-matrices, and set up our notation. All the material in this section is standard and we refer to \cite[Section 2.1]{Lu-Mou:mixed} and \cite[Section 2.4]{Lu-Mou:mixed} for details.

\subsection{Lie bialgebras and quasitriangular $r$-matrices} \label{subsec-ie-bial-quasi-tri}

Let $\g$ be a finite dimensional Lie algebra over a field $\IK$ of zero characteristic.  A {\it Lie bialgebra structure} on $\g$ is a map $\delta_\g: \g \to \wedge^2 \g$ satisfying the cocycle condition 
$$
\delta_\g[x,y] = [x, \delta_\g(y)] + [\delta_\g(x),y], \hs x,y \in \g,
$$
and such that the dual map $\delta_\g^*: \wedge^2 \g^* \to \g^*$ is a Lie bracket on $\g^*$, and one says that $(\g, \delta_\g)$ is a {\it Lie bialgebra}. A {\it quasitriangular $r$-matrix} on $\g$ is an element $r \in \g \otimes \g$ whose symmetric part $\Omega = \frac{1}{2}(r + r_{21})$ is $\g$-invariant, and which satisfies the {\it Classical Yang-Baxter Equation} in $U(\g)^{\otimes 3}$
$$
[r_{12}, r_{13}] + [r_{12}, r_{23}] + [r_{13}, r_{23}] = 0.
$$
A quasitriangular $r$-matrix $r \in \g \otimes \g$ defines a Lie bialgebra structure 
$$
\delta_r: \g \to \wedge ^2 \g, \hs \delta_r(x) = [x,r] \hs x \in \g,
$$
where $\delta_r(x)$ indeed lies in $\wedge^2 \g \subset \g \otimes \g$ for every $x \in \g$, since $\Omega$ is $\g$-invariant. In particular, $\delta_r$ only depends on the anti-symmetric part $\Lambda = \frac{1}{2}(r - r_{21})$ of $r$. A {\it quasitriangular Lie bialgebra} is a Lie bialgebra $(\g, \delta_\g)$ such that $\delta_\g = \delta_r$ for a quasitriangular $r$-matrix $r \in \g \otimes \g$, and one says that $r$ is a {\it quasitriangular structure} for $(\g, \delta_\g)$. We also denote quasitriangular Lie bialgebras as pairs $(\g, r)$.

\subsection{Poisson actions and twists} \label{subsec-poiss-act}

A Poisson algebra is a pair $(A, \{ \:, \:\})$ where $A$ is a commutative algebra over $\IK$, and $\{ \:, \:\}$ a $\IK$-linear Lie bracket on $A$ which is a derivation in each factor. A (right) {\it Poisson action} of a Lie bialgebra $(\g, \delta_\g)$ on a Poisson algebra $(A, \{ \:, \:\})$ is a Lie algebra morphism $\rho: \g \to \Der(A)$, where $\Der(A)$ is the Lie algebra of derivations of $A$, satisfying
\begin{equation} \label{poiss-ac-alg}
\rho(\delta_\g(x))(f \wedge g) = \rho(x)(\{f, g\}) - (\{\rho(x)(f), g\} + \{f, \rho(x)(g)\}), \hs x \in \g, \; f,g \in A.  
\end{equation}
When $\IK = \IR$ or $\IC$ and $A$ is the algebra of smooth or holomorphic functions on a real or complex Poisson manifold, \eqref{poiss-ac-alg} coincides with the usual notion of a Poisson action of a Lie bialgebra. A {\it twisting element} of $(\g, \delta_\g)$ is an element $t \in \wedge^2 \g$ satisfying 
$$
\delta_\g(t) + \frac{1}{2}[t,t] = 0,
$$
so that $\delta_{\g,t}: \g \to \wedge^2 \g$ defined by 
$$
\delta_{\g,t}(x) = \delta_\g(x) + [t,x], \hs x \in \g,
$$
is a Lie bialgebra structure on $\g$, and we call $(\g, \delta_{\g,t})$ the {\it twist} of $(\g, \delta_\g)$ by $t$. Given a Poisson action $\rho: \g \to \Der(A)$ of a Lie bialgebra $(\g, \delta_\g)$ on a Poisson algebra $(A, \{ \:, \:\})$ and a twisting element $t$ of $(\g, \delta_\g)$, the bracket $\{ \:, \:\}_t$ on $A$ given by 
$$
\{f,g\}_t = \{f,g\} -\rho(t)(f \wedge g), \hs f,g \in A, 
$$
is then a Poisson bracket, $\rho: \g \to \Der(A)$ is a Poisson action of $(\g, \delta_{\g,t})$ on $(A, \{ \:, \:\}_t)$, and one says that $(A, \{ \:, \:\}_t)$ is the {\it twist of  $(A, \{ \:, \:\})$ by $\rho(t)$}. Moreover, if $(\g, \delta_\g)$ has quasitriangular structure $r \in \g \otimes \g$, $r - t$ is a quasitriangular structure for $(\g, \delta_{\g,t})$. 

Let $\g$ be a Lie algebra, $r = \sum_i x_i \otimes y_i \in \g \otimes \g$ a quasitriangular $r$-matrix on $\g$, and $m \geq 1$ an integer. For $1 \leq j \leq m$ and $x \in \g$, we denote by $(x)_j \in \g^m$ the image of $x$ under the embedding of $\g$ into $\g^m$ as the $j$'th component. Then 
\begin{equation} \label{eq_Mix}
\Mix^m(r)  = \sum_{1 \leq k < l \leq m} \Mix_{k,l}^m(r) \in \g^m \wedge \g^m, \;\; \text{with} \;\; \Mix_{k,l}^m(r) = \sum_i (y_i)_k \wedge (x_i)_l \in \g^m \wedge \g^m,
\end{equation}
is a twisting element of the direct product Lie bialgebra $(\g, r)^m$, and 
$$
r^{(m)} = (r, \ldots, r) - \Mix^m(r) \in \g^m \otimes \g^m,
$$
is a quasitriangular $r$-matrix for the twist of $(\g, r)^m$ by $\Mix^m(r)$. The diagonal embedding 
$\diag_m: (\g, \delta_r) \hookrightarrow (\g^m, \delta_{r^{(m)}})$ is a morphism of Lie bialgebras, and we call $(\g^m, r^{(m)})$ the  {\it twisted $m$-fold product} of $(\g, r)$. Our definition of $r^{(m)}$ differs slightly from the corresponding definition given in \cite[Section 6]{Lu-Mou:mixed}. See \cite[Remark 6.4]{Lu-Mou:mixed} for an explanation of this discrepancy. When $\rho: \g \to \Der(A)$ is a Poisson action of $(\g, r)$ on a Poisson algebra $(A, \{ \:,\:\})$, the twist of the direct product Poisson algebra $(A, \{ \:,\:\})^{\otimes m}$ by $\rho^{(m)}(\Mix^m(r))$ is called a {\it mixed product Poisson structure} in \cite{Lu-Mou:mixed}, where 
\begin{equation} \label{eq-defn-rho^m}
\rho^{(m)}: \g^m \to \Der(A)^m, \hs \rho^{(m)}(x_1, \ldots, x_m) = (\rho(x_1), \ldots, \rho(x_m)), \hs x_i \in \g. 
\end{equation}


\section{A Poisson structure on products of principal affine spaces} \label{sec-prod-G/N}


\subsection{Semisimple Lie algebras and the standard quasitriangular $r$-matrix} \label{subsec-std-PL}

Let $\g$ be a complex semisimple Lie algebra. We recall the standard quasitriangular $r$-matrix $r_{st}$ on $\g$ and refer to  \cite[Section 4]{Lu-Mou:double-B-cell} for details.

Fix a Cartan subalgebra $\t \subset \g$, let $\triangle \subset \t^*$ be the root system of $\g$ with respect to $\t$, and for $\alpha \in \triangle$, let $\g_\alpha \subset \g$ be the corresponding root space. Fix a system $\triangle_+ \subset \triangle$ of positive roots and let $\n = \sum_{\alpha \in \triangle_+} \g_{\alpha}$, $\n_- = \sum_{\alpha \in \triangle_+} \g_{-\alpha}$ be the nilpotent radicals of the pair $(\b = \t + \n, \: \b_- = \t + \n_-)$ of opposite Borel subalgebras of $\g$. Let $\lara_\g$ be the Killing form of $\g$ and denote by $\lara$ the bilinear form on both $\t$ and $\t^*$ induced by $\lara_\g$. For every positive root $\alpha \in \triangle_+$, choose root vectors $e_\alpha \in \g_\alpha$, $e_{-\alpha} \in \g_{-\alpha}$, such that $\la e_\alpha, e_{-\alpha} \ra_\g = 1$. Let $\{h_i\}_{i =1}^k$ be a basis of $\t$ orthonormal with respect to $\lara$, and let 
$$
r_0 = \frac{1}{2} \sum_{i=1}^k h_i \otimes h_i \in \t \otimes \t, \hs \text{and} \hs r_{st} = r_0 + \sum_{\alpha \in \triangle_+} e_{-\alpha} \otimes e_\alpha \in \g \otimes \g. 
$$
Then $r_{st}$ is a quasitriangular $r$-matrix, called the {\it standard quasitriangular $r$-matrix} associated to the pair $(\b, \b_-)$, and its antisymmetric part
$$
\Lambda_{st} =  \frac{1}{2}\sum_{\alpha \in \triangle_+} e_{-\alpha} \wedge e_\alpha \in \wedge^2 \g 
$$
is the {\it standard skew-symmetric $r$-matrix}  associated to $(\b, \b_-)$.

We label the simple roots in $\triangle_+$ by $\{\alpha_i : i = 1,  \ldots, k\}$. For $1 \leq i \leq k$, let $\varpi_i \in \t^*$ be the fundamental weight corresponding to $\alpha_i$,  and let 
$$
\P_+ = \{\sum_{i =1}^k n_i \varpi_i : n_i \in \IN_0 \}
$$
be the free monoid of dominant integral weights. Let $W$ be the Weyl group of $G$ and $w_0 \in W$ the longest element in $W$. For $\varpi \in \P_+$, we denote by $V(\varpi)$ the irreducible representation of $\g$ of highest weight $\varpi$, and recall that one has $V(\varpi)^* = V(-w_0(\varpi))$ as $\g$-modules.

\subsection{Principal affine spaces}

Let $G$ be the connected and simply connected complex semisimple algebraic group with Lie algebra $\g$, and let $B \subset G$ be the Borel subgroup with Lie algebra $\b \subset \g$. Let $N \subset B$ be the unipotent radical of $B$ and let $T \subset B$ be the maximal torus of $G$ with Lie algebra $\t \subset \b$. We recall in this subsection the representation theory of $\IC[N \backslash G]$. Our reference is \cite[Chapter 12]{goodman-wallace}. 

Ample line bundles on the flag variety $B \backslash G$ are parametrized by $\P_+$, where for $\varpi \in \P_+$, $L(\varpi)$ is the line bundle over $B \backslash G$ defined as the quotient of $\IC \times G$ by the left action of $B$ given by 
$$
b \cdot (z, g) = (b^{-\varpi}z, bg), \hs g \in G, \; b \in B, \; z \in \IC,
$$
where $b^{\varpi} = e^{\varpi(x)}$ if $b = \exp(x)n$, with $n \in N$ and $x \in \t$. Via pull back by the quotient map $G \to B \backslash G$, one can identify the global sections of $L(\varpi)$ with 
$$
\IC[G]^{\varpi} : = \{f \in \IC[G] : f(bg) = b^{-\varpi} f(g), \; g \in G, \; b \in B\}. 
$$
Under the left action of $G$ on $\IC[G]$ given by $(g_1 \cdot f)(g_2) = f(g_2g_1)$, $g_i \in G$,  $\IC[G]^{\varpi}$ is $G$-invariant and isomorphic to $V(-w_0(\varpi))  \cong V(\varpi)^*$. Indeed, for any finite dimensional representation $V$ of $G$, one has the {\it matrix coefficients}
$$
c^\sV_{\xi, v} \in \IC[G], \hs c^\sV_{\xi, v}(g) = \la \xi, \: g^{-1} \cdot v \ra = \la g \cdot \xi, \: v \ra, \hs g \in G,
$$
defined for $v \in V$ and $\xi \in V^*$, and the map 
\begin{equation} \label{eq-Phi-varpi}
\Phi_\varpi: V(\varpi) \to \IC[G]^{-w_0(\varpi)}, \hs \Phi_\varpi(\xi) =c^{\scriptscriptstyle V(-w_0(\varpi))}_{\xi, v}, \hs \xi \in V(\varpi), 
\end{equation}
is a $G$-equivariant isomorphism, where $v \in V(-w_0(\varpi)) \cong V(\varpi)^*$ is any fixed highest weight vector. The {\it principal affine space} $N \backslash G$ has affine coordinate ring 
$$
\IC[N \backslash G] = \{f \in \IC[G] : f(ng) = f(g), \; g \in G, \; n \in N\},
$$
and one has 
$$
\IC[N \backslash G] = \bigoplus_{\varpi \in \P_+} \IC[G]^{\varpi}. 
$$

\subsection{A Poisson structure on products of principal affine spaces} \label{subsec-poiss-products}

Let $r_{st}$ be the quasitriangular $r$-matrix defined in $\S$\ref{subsec-std-PL}. For any integer $m \geq 1$, one has the $m$'th {\it polyuble Poisson Lie group} $(G^m, \pi_{st}^{(m)})$ \cite{FR2, Lu-Mou:mixed}, where 
\begin{align} \label{eq-brack-r^(m)}
\pi_{st}^{(m)} =  (r_{st}^{(m)})^L - (r_{st}^{(m)})^R = (\Lambda_{st}^{(m)})^L - (\Lambda_{st}^{(m)})^R, 
\end{align}
and
$$
\Lambda_{st}^{(m)} = (\Lambda_{st}, \ldots, \Lambda_{st}) - \Mix^m(r_{st})  \in \wedge^2 \g^m
$$ 
is the antisymmetric part of $r_{st}^{(m)}$. In particular $(G^m, \pi_{st}^{(m)})$ has Lie bialgebra $(\g^m, r_{st}^{(m)})$, thus 
$$
\sigma_m: \g^m \oplus \g^m \to \XX^1(G^m), \hs \sigma_m(y,x) = y^L-x^R, \hs x,y \in \g^m, 
$$
is a Poisson action of $(\g^m, r_{st}^{(m)}) \oplus (\g^m, -r_{st}^{(m)})$ on $(\IC[G^m], \{ \:, \:\}_{r_{st}}^{(m)})$, where $\{ \:, \:\}_{r_{st}}^{(m)}$ is the Poisson bracket induced by $\pi_{st}^{(m)}$. 

We first treat the $m = 1$ case. For $\varpi, \lambda \in \P_+$, let 
$$
p_{\varpi, \lambda}: V(\varpi) \otimes V(\lambda) \cong V(\lambda) \otimes V(\varpi) \to V(\varpi +\lambda)
$$ 
be the projection along the $\g$-submodules of highest weight less than $\varpi + \lambda$.

\begin{pro} \label{pro-G/N}
1) For any $f,g \in \IC[N \backslash G]$, one has 
\begin{equation} \label{pi_G/N}
\{f,g\}^{(1)}_{r_{st}} = \Lambda_{st}^L(f \wedge g) \in \IC[N \backslash G],
\end{equation}
i.e $\pist^{(1)}$ descends to a well defined Poisson structure $\pi^{(1)}$ on $N \backslash G$. 

2) Let $\varpi, \lambda \in \P_+$ and fix highest vectors $v \in V(\varpi)^*$, $w \in V(\lambda)^*$, and recall the isomorphisms defined in \eqref{eq-Phi-varpi}. For $\xi \in V(\varpi)$ and $\mu \in V(\lambda)$, one has 
\begin{equation} \label{pi_G/N-repr}
\{\Phi_\varpi(\xi), \Phi_\lambda(\mu)\}^{(1)}_{r_{st}} = \Phi_{\varpi + \lambda}(p_{\varpi, \lambda}(\Lambda_{st}\cdot (\xi \otimes \mu))). 
\end{equation}
In particular, $\IC[N \backslash G]$ is a $\P_+$-graded Poisson subalgebra of $(\IC[G], \{ \:, \:\}^{(1)}_{r_{st}})$. 
\end{pro}
\begin{proof}
1) Since $\Lambda_{st} \in \n_- \wedge \n$, one clearly has $\Lambda_{st}^R(f \wedge g) = 0$ for any $f,g \in \IC[N \backslash G]$, which proves \eqref{pi_G/N}.

2) Since $v \otimes w$ is a highest weight vector in $(V(\varpi) \otimes V(\lambda))^*$, for $\xi \in V(\varpi)$, $\mu \in V(\lambda)$, one has 
\begin{align*}
\Phi_\varpi(\xi) \Phi_\lambda(\mu) & = c^{\scriptscriptstyle V(-w_0(\varpi))}_{\xi, v} c^{\scriptscriptstyle V(-w_0(\lambda))}_{\mu, w} = c^{\scriptscriptstyle V(-w_0(\varpi)) \otimes V(-w_0(\lambda))}_{\xi \otimes \mu, \; v \otimes w}   \\  
    & = c^{\scriptscriptstyle V(-w_0(\varpi + \lambda))}_{p_{\varpi, \lambda}(\xi \otimes \mu), \;\; v \otimes w} = \Phi_{\varpi + \lambda}(p_{\varpi, \lambda}(\xi \otimes \mu)),
\end{align*}
and \eqref{pi_G/N-repr} is now a straightforward consequence of \eqref{pi_G/N}.
\end{proof}

We denote by $\{ \:, \:\}^{(1)}$ the restriction of $\{ \:, \:\}^{(1)}_{r_{st}}$ to $\IC[N \backslash G]$. Let $\tilde{\g} =\g \oplus \t$ as a direct sum Lie algebra and let 
\begin{equation} \label{tilder-r_0}
\tilde{r}_{st} = (r_{st}, 0) - (0, r_0) \in \tilde{\g} \otimes \tilde{\g}. 
\end{equation}
That is,  the quasitriangular Lie bialgebra $(\tilde{\g}, \tilde{r}_{st})$ is the direct sum of the quasitriangular Lie bialgebras $(\g, r_{st})$ and $(\t, -r_0)$. For $x \in \t$ and $f \in \IC[G]^{\varpi} \subset \IC[N \backslash G]$, one has $x^R(f) = - \la \varpi, x \ra  f$. Thus the action $\sigma_1$ restricts to a Poisson action 
\begin{equation} \label{eq-tilde-rho}
\rho: \tilde{\g} \to \XX^1(N \backslash G), \hs \rho(y, x) = y^L -x^R, \hs x \in \t, \; y \in \g, 
\end{equation}
of $(\tilde{\g}, \tilde{r}_{st})$ on  $(\IC[N \backslash G], \{ \:, \:\}^{(1)})$. 

We now treat the case of a general $m \geq 1$. Let $(\pi^{(1)})^{m}$ be the direct product Poisson structure on $(N \backslash G)^m$ and let $\rho^{(m)}: \tilde{\g}^m \to \XX^1(N \backslash G)^m$ be as in \eqref{eq-defn-rho^m}, so that by $\S$\ref{subsec-poiss-act} one has the Poisson structure 
$$
\pi^{(m)} = (\pi^{(1)})^{m} - \rho^{(m)}(\Mix^m(\tilde{r}_{st}))
$$
on $(N \backslash G)^m$, with Poisson bracket $\{ \:, \:\}^{(m)}$ on $\IC[(N \backslash G)^m] \cong \IC[N \backslash G]^{\otimes m}$.

\begin{thm} \label{main-thm-G/N}
1) The diagonal action of $G$ on $(N \backslash G)^m$, 
$$
(N.g_1, \ldots, Ng_m) \cdot g = (N.g_1g, \ldots, N.g_mg), \hs g, g_i \in G,
$$
is a Poisson action of the Poisson Lie group $(G, \pi_{st}^{(1)})$ on $((N \backslash G)^m, \pi^{(m)})$, or equivalently, 
$$
\phi: \g \to \XX^1(N \backslash G)^m, \hs \phi(y) = (y^L, \ldots, y^L), \hs y \in \g, 
$$
is a Poisson action of $(\g, r_{st})$ on $(\IC[N \backslash G]^{\otimes m}, \{ \:, \:\}^{(m)})$. 

2) The algebra $\IC[(N \backslash G)^m] \cong \IC[N \backslash G]^{\otimes m}$ is a Poisson subalgebra of $(\IC[G^m], \{ \:, \: \}_{r_{st}}^{(m)})$ and one has 
$$
\{ f, g\}_{r_{st}}^{(m)} = \{f, g\}^{(m)}, \hs f, g \in \IC[N \backslash G]^{\otimes m}, 
$$
i.e $\pi_{st}^{(m)}$ projects to $\pi^{(m)}$ on $(N \backslash G)^m$. 
\end{thm}
\begin{proof}
Part 1) is clear since $(\g, \delta_{r_{st}}) \hookrightarrow (\tilde{\g}, \delta_{\tilde{r}_{st}})$, $y \mapsto (y, 0)$, and $\diag_m: (\tilde{\g}, \delta_{\tilde{r}_{st}})  \hookrightarrow (\tilde{\g}^m, \delta_{\tilde{r}_{st}^{(m)}})$, defined as in $\S$\ref{subsec-poiss-act} are Lie bialgebra morphisms. As for part 2), let $r_{\pm} = r_{st} - r_0$. Since $r_{\pm} \in \n_- \otimes \n$, for $f,g \in \IC[N \backslash G]^{\otimes m}$ one has 
$$
\Mix^m(r_{\pm})^R (f \otimes g) = 0,
$$
hence 
\begin{align*}
\{ f, g\}_{r_{st}}^{(m)}  & = (\Lambda_{st}^{(m)})^L (f \otimes g) -  (\Lambda_{st}^{(m)})^R (f \otimes g) \\
   & =   (\Lambda_{st}^{(m)})^L (f \otimes g) - \Mix^m(r_0)^R (f \otimes g)  \\
   & =  \{f, g\}^{(m)}. 
\end{align*}
\end{proof}

\begin{rem} 
The Poisson bracket $\{ \:, \:\}^{(m)}$ is given explicitly as follows. For $\varpi, \lambda \in \P_+$, $f \in \IC[G]^\varpi$, $g \in \IC[G]^\lambda$, and $1 \leq l < t \leq m$, recalling our notational conventions of $\S$\ref{subsec-nota}, one has 
\begin{align*}
\{f_l, g_l \}^{(m)} & = \left(\{f,g\}^{(1)}\right)_l, \notag  \\
\{f_l, g_t \}^{(m)} & =   - \la \varpi \otimes \lambda, r_0\ra f_l g_t - \Mix^m(r_{st})^L (f_l \wedge g_t) \notag   \\
   & = - \la \varpi \otimes \lambda, r_0\ra f_l g_t \\
   & \hs - \frac{1}{2} \sum_{i=1}^k (h_i^L(f))_l(h_i^L(g))_t -  \sum_{\alpha \in \triangle_+} (e_{-\alpha} ^L(f))_l(e_\alpha^L(g))_t.   
\end{align*}
In particular, $\{ \:, \:\}^{(m)}$ is a $(\P_+)^m$-graded Poisson bracket on $\IC[N \backslash G]^{\otimes m}$, i.e
\begin{equation}  \label{rem-graded-br}
\{\IC[G]^{\ulamb}, \: \IC[G]^{\umu} \}^{(m)} \subset \IC[G]^{\ulamb + \umu},
\end{equation}
where for any $\ulamb = (\lambda_1, \ldots, \lambda_m) \in (\P_+)^m$, 
$$
\IC[G]^{\ulamb}: = \IC[G]^{\lambda_1} \otimes \cdots \otimes \IC[G]^{\lambda_m} \subset \IC[N \backslash G]^{\otimes m}.
$$
\hfill $\diamond$
\end{rem}

\subsection{Flat Poisson connections} \label{subsec-poiss-conn}

Let $L$ be a line bundle over a complex Poisson variety $(X, \{ \:,\:\}_\sX)$. Recall from \cite[Section 1]{Po} that a {\it flat Poisson connection} on $L$ is a linear map 
$$
\nabla: \O_\sX \otimes \O(L) \to \O(L), \hs f \otimes s \mapsto \nabla_f s, \hs f \in \O_\sX, \; s \in \O(L), 
$$
where $\O(L)$ is the sheaf of sections of $L$, which is a derivation in the first argument, and satisfies 
\begin{align*}
\nabla_f \: gs & = \{f,g\}_\sX s + g \nabla_f s,     \\
\nabla_{\{f,g\}_\sX} \:s &  = \nabla_f \nabla_gs - \nabla_g \nabla_fs, \hs f,g \in \O_\sX, \; s \in \O(L). 
\end{align*}

We continue with the discussion in $\S$\ref{subsec-poiss-products}. By \cite[Section 7]{Lu-Mou:mixed}, $\pi_{st}^{(m)}$ descends to a well defined Poisson structure $\pi_m$ on $(B \backslash G)^m$, thus $\pi^{(m)}$ descends to $\pi_m$ under the natural projection $(N \backslash G)^m \to (B \backslash G)^m$. For $\ulamb = (\lambda_1, \ldots, \lambda_m) \in (\P_+)^m$, let 
$$
L(\ulamb) = L(\lambda_1) \boxtimes \cdots \boxtimes L(\lambda_m)
$$
be the corresponding ample line bundle on $(B \backslash G)^m$, and identify the global sections of $L(\ulamb)$ with $\IC[G]^{\ulamb}$. Any local function $f$ on $(B \backslash G)^m$ can be written as  $f = f_1/f_2$, with $f_j \in \IC[G]^{\umu}$, for a $\umu \in (\P_+)^m$. Thus for $\ulamb \in (\P_+)^m$, one can define a Poisson connection $\nabla^{\ulamb}$ on $L(\ulamb)$, whose action on global sections is given by 
\begin{equation} \label{eq-flat-conn}
(\nabla^{\ulamb})_{f_1/f_2}\:s = \left\{ \frac{f_1}{f_2}, s\right\}^{(m)}, \hs f_j \in \IC[G]^{\umu}, \hs s \in \IC[G]^{\ulamb}. 
\end{equation}

\begin{pro} \label{pro-flat-conn}
For $\ulamb \in (\P_+)^m$, $\nabla^{\ulamb}$ is a well-defined flat Poisson connection on $L(\ulamb)$.
\end{pro}
\begin{proof}
Indeed, by \eqref{rem-graded-br} one has $(\nabla^{\ulamb})_fs \in \O(L(\ulamb))$ for any local function $f$ on $(B \backslash G)^m$ and $s \in \O(L(\ulamb))$, and it is clear that \eqref{eq-flat-conn} defines a flat Poisson connection. 
\end{proof}

\begin{rem}
By \cite[Proposition 5.2]{Po}, a flat Poisson connection $\nabla$ on a line bundle $L$ over a Poisson variety $(X,  \{ \:,\:\}_\sX)$ defines a graded Poisson bracket on the corresponding graded algebra $\oplus_{n \geq 0} H^0(X, L^{\otimes n})$. For $\ulamb \in (\P_+)^m$, it is easily checked that the graded Poisson bracket on 
$$
S(\ulamb) = \bigoplus_{n \geq 0} \IC[G]^{n \ulamb} \subset \IC[N \backslash G]^{\otimes m}  
$$
defined by $\nabla^{\ulamb}$ coincides with $\{ \:, \:\}^{(m)}$. 
\end{rem}


\section{Strongly coisotropic subalgebras of Lie subalgebras} \label{sec-strong-coiso-liebialg}

In view of Proposition \ref{pro-G/N} and Theorem \ref{main-thm-G/N} we introduce the definition of a strongly coisotropic subalgebra of a Lie bialgebra. Let $(\g, \delta_\g)$ be a Lie bialgebra and denote by $[\: , \: ]_{\g^*}$ the Lie bracket on $\g^*$ which is the dual of the cocycle map $\delta_\g$. For a vector space $\u \subset \g$, let $\u^0 \subset \g^*$ be its annihilator. 

\begin{defn} \label{defn-strong-co}
A Lie subalgebra $\u \subset \g$ is said to be a {\it strongly coisotropic subalgebra} of $(\g, \delta_\g)$ if 
$$
\delta_\g(\u) \subset \g \otimes [\u,\u] + [\u,\u] \otimes \g,
$$ 
or equivalently, if 
$$
[[\u,\u]^0, [\u,\u]^0]_{\g^*} \subset \u^0.
$$
\end{defn}

\begin{rem}
Recall from \cite{lu:thesis} that a Lie subalgebra $\u \subset \g$ is said to be a coisotropic subalgebra of $(\g, \delta_\g)$ if $\delta_\g(\u) \subset \g \otimes \u + \u \otimes \g$, thus a strongly coisotropic subalgebra is coisotropic. Note also that if $\u$ is strongly coisotropic, then $[\u,\u]$ is itself a coisotropic subalgebra of $(\g, \delta_\g)$. Hence $\u$ is strongly coisotropic if and only if the induced Lie bracket on $[\u, \u]^0/\u^0$ is trivial. 
\hfill $\diamond$
\end{rem}

Let $(A, \{ \:, \:\})$ be a Poisson algebra and $\rho: \g \to \Der(A)$ a Poisson action of a Lie bialgebra $(\g, \delta_\g)$. If $\u \subset \g$ is a Lie subalgebra, let $\Ch(\u) = \{\zeta \in \u^* : \la \zeta, [\u, \u] \ra = 0 \ra \}$, and for $\zeta \in \Ch(\u)$, let 
$$
A^{\zeta} = \{f \in A : \rho(x)(f) = \la \zeta, x \ra f, \hs x \in \u \} 
$$
be the weight space in $A$ of weight $\zeta$. Let $A^\u \subset A$ be the subalgebra of $\u$-semi invariant elements of $A$, that is 
\begin{align*}
A^\u & = \bigoplus_{\zeta \in \P_\sA(\u)} A^{\zeta} \: \subset A, \hs \text{where}    \\
\P_\sA(\u) & = \{ \zeta \in \Ch(\u): A^{\zeta} \neq 0 \}.
\end{align*}

\begin{pro} \label{main-pro-1}
If $\u$ is a strongly coisotropic subalgebra of $(\g, \delta_\g)$, 
$$
\{A^{\zeta_1}, A^{\zeta_2}\} \subset A^{\zeta_1 + \zeta_2}, \hs \zeta_1, \zeta_2 \in \P_\sA(\u), 
$$
i.e $A^\u$ is a Poisson subalgebra of $(A, \{ \:, \:\})$ graded by $\P_\sA(\u)$. 
\end{pro}
\begin{proof}
Let $\zeta_1, \zeta_2 \in \P_\sA(\u)$, $f \in A^{\zeta_1}$, and $g \in A^{\zeta_2}$. As $\delta_\g(\u) \subset \g \otimes [\u,\u] + [\u,\u] \otimes \g$, one has 
$$
\rho(\delta_\g(x))(f \wedge g) = 0,  \hs x \in \u, 
$$
and so by \eqref{poiss-ac-alg}, 
$$
\rho(x)(\{f, g\}) = \{\rho(x)(f), g\} + \{f, \rho(x)(g)\} = (\zeta_1 + \zeta_2)(x)\{f, g\}, \hs x \in \u, 
$$
hence 
$$
\{A^{\zeta_1}, A^{\zeta_2}\} \subset A^{\zeta_1 + \zeta_2}. 
$$
\end{proof}

\begin{rem}
It is shown in \cite{lu:thesis} that if $\u$ is a coisotropic subalgebra of $(\g, \delta_\g)$, the $\u$-invariant elements $A^{0} = \{f \in A: \rho(x)f = 0, \; x \in \u \}$ is a Poisson subalgebra of $(A, \{ \:, \:\})$. Hence one should think of strong coisotropicity as a generalization of coisotropicity to the graded setting. 
\hfill $\diamond$
\end{rem}

The following Lemma \ref{lem-strong_co-r} is straightforward.

\begin{lem} \label{lem-strong_co-r}
Let $r \in \g \otimes \g$ be a quasitriangular $r$-matrix on a Lie algebra $\g$. If $\u \subset \g$ is a subalgebra such that 
\begin{equation} \label{eqn-strong_cor-r}
r \in \u \otimes \u + \g \otimes [\u,\u] + [\u,\u] \otimes \g,
\end{equation}
then $\u$ is a strongly coisotropic subalgebra of $(\g, r)$. In particular, if \eqref{eqn-strong_cor-r} holds, $\u^m$ is a strongly coisotropic subalgebra of $(\g^m, r^{(m)})$ for any $m \geq 1$. 
\end{lem}

\begin{exas}
1) Assume the setting of $\S$\ref{subsec-std-PL}. Since $r_{st}$ and $\b$ satisfy \eqref{eqn-strong_cor-r}, for any Lie subalgebra $\p$ of $\g$ containing $\b$, $\p^m$ is a strongly coisotropic subalgebra of $(\g^m, r_{st}^{(m)})$ for any $m \geq 1$. The conclusion of Theorem \ref{main-thm-G/N} that $\IC[N \backslash G]^{\otimes m}$ is a graded Poisson subalgebra of $(\IC[G^m], \{ \:,\:\}_{r_{st}}^{(m)})$ now follows from Proposition \ref{main-pro-1}, as 
$$
x \mapsto -x^R, \hs x \in \g^m,
$$ 
is a Poisson action of $(\g^m, -r_{st}^{(m)})$ on $(\IC[G^m], \{ \:,\:\}_{r_{st}}^{(m)})$. 

2) More generally, one easily checks that any quasitriangular Belavin-Drinfeld $r$-matrix satisfies \eqref{eqn-strong_cor-r} with $\u = \b$, where we refer to \cite{BD} for the theory of such $r$-matrices. Hence for any $m \geq 1$ and any Belavin-Drinfeld $r$-matrix $r$, $\IC[N \backslash G]^{\otimes m}$ is a $(\P_+)^m$-graded Poisson subalgebra of $(\IC[G^m], \{ \:,\:\}_r^{(m)})$, where $\{ \:,\:\}_r^{(m)}$ is the Poisson bracket induced by $\pi_r^{(m)} = (r^{(m)})^L - (r^{(m)})^R$.  
\hfill $\diamond$
\end{exas}


\section{Hopf algebras and QUE algebras} \label{sec-hopf}

In this section we recall basic facts about Hopf algebras, QUE algebras and actions thereof, and refer to any book on quantum groups, such as \cite{chari-pressley, drin:quantum.groups, joseph, kassel, koro-soibelman}, for details.

\subsection{Hopf algebras}\label{subsec-hopf-alg}

For a Hopf algebra $H$ over a field $\IK$ of zero characteristic, we denote by $\mu_\sH: H \otimes H \to H$, $1_\sH \in H$, $\Delta_\sH:  H \to H \otimes H$, $\varepsilon_\sH:  H \to \IK$, and $S_\sH:  H \to H$ respectively the multiplication, unit, comultiplication, counit, and antipode of $H$, and if no confusion is possible, we write the multiplication as concatenation, $\mu_\sH(h_1 \otimes h_2) = h_1h_2$. Recall that a quasitriangular $R$-matrix on a Hopf algebra $H$ is an invertible element $R \in H \otimes H$ satisfying 
\begin{align}
&\Delta_\sH^{op} = R \Delta_\sH R^{-1},   \label{almost-com} \\
& (I_\sH \otimes \Delta_\sH)R = R_{13}R_{12}, \hs \hs (\Delta_\sH \otimes I_\sH)R = R_{13}R_{23}, \label{hexagon}
\end{align}
where $I_\sH$ is the identity map of $H$, and one says that $(H, R)$ is a quasitriangular Hopf algebra. We mention the relation 
\begin{align}
(\varepsilon_\sH \otimes I_\sH)(R) & = 1_\sH  = (I_\sH \otimes \varepsilon_\sH)(R)  \label{eqn-eps-I-R}
\end{align}
which will be used in $\S$\ref{sec-twisted}. 

For $m \geq 1$ we denote by  $\Delta_\sH^{(m)}: H \to H^{\otimes m}$ the $m$-fold comultiplication. Thus $\Delta_\sH^{(1)} = I_\sH$ and $\Delta_\sH^{(2)} = \Delta_\sH$  by convention, and if $m \geq 3$,  $\Delta_\sH^{(m)} = (\Delta_\sH \otimes I_\sH^{\otimes m-1})\Delta_\sH^{(m-1)}$.

\subsection{QUE algebras} \label{subsec-QUE}

A {\it deformation Hopf algebra} is a topologically free $\kh$-algebra $H$ equipped with maps $\Delta_\sH:H \to H \hat{\otimes} H$, $\varepsilon_\sH: H \to \kh$, $S_\sH: H \to H$ satisfying axioms similar to those of a Hopf algebra, but where one must replace the tensor product $H \otimes H$ by the completed tensor product $H \hat{\otimes} H$. A {\it quantized universal enveloping algebra} (QUE algebra) is a deformation Hopf algebra $H$ such that there is a finite dimensional Lie algebra $\g$ with $H/\hslash H$ isomorphic to $U(\g)$ as a Hopf algebra. There is then a natural bialgebra structure $\delta_\g$ on $\g$ given by 
$$
\delta_\g(x) = \frac{1}{\hslash} (\Delta_\sH(\tilde{x}) - \Delta_\sH^{op}(\tilde{x})) \; \; {\rm mod} \: \hslash, \hs x \in \g,
$$
where $\tilde{x} \in H$ is any element such that $\tilde{x} + \hslash H = x$,  and one says that $H$ is a {\it quantization} of $(\g, \delta_\g)$. Similarly, any quasitriangular $R$-matrix $R \in H \hat{\otimes} H$ defines a quasitriangular structure $r \in \g \otimes \g$ on $(\g, \delta_\g)$ by 
$$
r = \frac{1}{\hslash}(R - 1_\sH \otimes 1_\sH)\; \; {\rm mod} \: \hslash,
$$
and one says that $R$ is a {\it quantization} of $r$. By Etingof and Kazhdan \cite{quant-li-bialg-I}, any quasitriangular Lie bialgebra can be functorially quantized to a quasitriangular QUE algebra.

\subsection{Module algebras and twists of Hopf algebras} \label{subsec-module-alg}

Let $H$ be a Hopf algebra, $A$ an associative algebra, and suppose that $A$ is a left module over $H$. One says that $A$ is a {\it left $H$-module algebra} if 
\begin{align*}
h \cdot (fg) & = \mu_\sA (\Delta_\sH(h) \cdot (f \otimes g)), \hs \hs h \in H, \;  f,g \in A,   \\
h \cdot 1_\sA & = \varepsilon_\sH(h)1_\sA, \hs \hs h \in H, 
\end{align*}
and right $H$-module algebras are similarly defined. One says that $A$ is an {\it $H$-bimodule algebra} if $A$ is a bimodule over $H$ as well as a left and right $H$-module algebra. A {\it twisting element} of $H$ is an invertible element $J \in H \otimes H$ satisfying
\begin{align}
(\Delta_\sH \otimes I_\sH)(J)J_{12} & = (I_\sH \otimes \Delta_\sH)(J)J_{23}, \label{eq-twist-defn1}  \\
(\varepsilon_\sH \otimes I_\sH)(J) & = (I_\sH \otimes \varepsilon_\sH)(J) = 1_\sH. \label{eq-twist-defn2} 
\end{align}
Then $J$ defines a new Hopf algebra structure on the underlying algebra $(H, \mu_\sH, 1_\sH)$ of $H$ by 
\begin{align*}
\Delta_{(\sH)_\sJ} & = J^{-1}\Delta_\sH J,    \\
S_{(\sH)_\sJ} & = Q^{-1}S_\sH Q, \hs \text{where} \hs Q = \mu_\sH((S_\sH \otimes I_\sH)(J)),   \\
\varepsilon_{(\sH)_\sJ} & = \varepsilon_{\sH},
\end{align*}
and we denote by $(H)_\sJ$ this newly obtained Hopf algebra, called the {\it twist of $H$ by $J$}. If $R \in H \otimes H$ is a quasitriangular $R$-matrix for $H$, $J^{-1}_{21}RJ$ is a quasitriangular $R$-matrix for $(H)_\sJ$. 

The following Lemma \ref{pro-twist-module} is a straightforward consequence of \eqref{eq-twist-defn1} and \eqref{eq-twist-defn2}.  

\begin{lem} \label{pro-twist-module}
Let $A$ be a left $H$-module algebra. The map $(\mu_\sA)_\sJ: A \otimes A \to A$ defined by 
\begin{equation} \label{defn-mub_F}
(\mu_\sA)_\sJ(f \otimes g) = \mu_\sA(J \cdot (f \otimes g)), \hs \hs f,g \in A, 
\end{equation}
is an associative product on $A$, and $(A, \!(\mu_\sA)_\sJ, \!1_\sA)$ is a left $(H)_\sJ$-module algebra. 
\end{lem}

\begin{defn} \label{defn-twist-module}
In the context of Lemma \ref{pro-twist-module}, one says that $(A, (\mu_\sA)_\sJ, 1_\sA)$ is the {\it twist} of $A$ by $J \in H \otimes H$, and denote by $(A)_\sJ$ this newly obtained algebra. 
\end{defn}

Module algebras and twisting of deformation Hopf algebras are similarly defined.

\subsection{Quantization of Poisson actions}

By a {\it deformation quantization algebra} (DQ algebra) we mean in this paper a topologically free associative $\kh$-algebra $A$ such that $A_0 = A/\hslash A$ is a commutative $\IK$-algebra. The multiplication on $A$ induces on $A_0$ a Poisson bracket given by 
$$
\{f,g\} = \frac{1}{\hslash} (\tilde{f}\tilde{g} - \tilde{g}\tilde{f}) \; \; {\rm mod} \: \hslash, \hs f, g \in A_0,
$$
where $\tilde{f}, \tilde{g} \in A$ are any elements such that $\tilde{f} + \hslash A = f$ and $\tilde{g} + \hslash A = g$, and one says that $A$ is a {\it quantization} of $(A_0, \{ \:,\:\})$.

\begin{lem} \label{lem-equiv-quant} 
Let $U_\hslash(\g)$ be a QUE algebra quantizing a Lie bialgebra $(\g, \delta_\g)$, $A$ a DQ algebra quantizing a Poisson algebra $(A_0, \{ \:,\:\})$, and assume that $A$ is a left $U_\hslash(\g)$-module algebra. Then 
\begin{equation} \label{eq-equiv-quant}
\rho(x)(f) = \tilde{x} \cdot \tilde{f}  \; \; {\rm mod} \: \hslash, \hs x \in \g, \; f \in A_0,
\end{equation}
where $\tilde{x} \in U_\hslash(\g)$ and $\tilde{f} \in A$ are such that $\tilde{x} + \hslash U_\hslash(\g) = x$ and $\tilde{f} + \hslash A = f$, defines a Lie algebra morphism $\rho: \g \to \Der(A_0)$ which is a Poisson action of $(\g, \delta_\g)$ on $(A_0, \{ \:, \})$. 
\end{lem}
\begin{proof}
Since $U(\g) \cong U_\hslash(\g)/\hslash U_\hslash(\g)$ as a Hopf algebra, the $U_\hslash(\g)$-module algebra structure on $A$ induces a $U(\g)$-module algebra on $A_0$, that is the map $\rho$ defined in \eqref{eq-equiv-quant} is a Lie algebra action. 

Let $\Delta$ be the comultiplication in $U_\hslash(\g)$. As $U_\hslash(\g)$ and $A$ are topologically free, one can assume that $U_\hslash(\g) = U(\g)[[\hslash]]$ and $A = A_0[[\hslash]]$ as $\kh$-modules. Then for $x \in \g$ and $f, g \in A_0$, the $\hslash^0$ term of the identity
$$
x \cdot \mu_\sA(f \otimes g - g \otimes f) = \mu_\sA(\Delta(x) \cdot (f \otimes g - g \otimes f)) 
$$
vanishes, and the $\hslash^1$ term reads as  
$$
\rho(x)(\{f,g\}) = \{\rho(x)(f), g \} + \{f, \rho(x)(g)\} + \rho(\delta_\g(x))(f \wedge g),
$$
i.e $\rho$ is a Poisson action of $(\g, \delta_\g)$ on $(A_0, \{ \: , \: \})$. 
\end{proof}


\section{Twisted products of quasitriangular Hopf algebras} \label{sec-twisted}

In this section, we describe a Hopf algebra construction analogous to the construction of the twisted $m$-fold product $(\g^m, r^{(m)})$ of a quasitriangular Lie bialgebra $(\g, r)$. When applied to a quantization of $(\g, r)$, one obtains a quantization of $(\g^m, r^{(m)})$, and we construct a quantization of mixed product Poisson structures.

\subsection{Twisted products of quasitriangular Hopf algebras}

Recall that the structure maps of the tensor product of two Hopf algebras $H_1, H_2$ are given by 
$$
\Delta_{\scriptscriptstyle H_1\otimes H_2} = \tau_{(23)} (\Delta_{\sH_1} \otimes \Delta_{\sH_2}), \hs S_{\scriptscriptstyle H_1 \otimes H_2} = S_{\sH_1} \otimes S_{\sH_2}, \hs \varepsilon_{\scriptscriptstyle H_1 \otimes H_2} = \varepsilon_{\sH_1} \otimes \varepsilon_{\sH_2},
$$
where $\tau_{(23)}$ is the permutation of the second and third tensor factors. The following Lemma \ref{lem-STS} was proven in \cite{STS-quantum}. 

\begin{lem} \label{lem-STS} \cite[Section 2]{STS-quantum}
Let $(H, R)$ be a quasitriangular Hopf algebra. Then $R_{23} \in H^{\otimes 4}$ is a twisting element of $H^{\otimes 2}$, and the comultiplication 
$$
\Delta_\sH: H \to (H^{\otimes 2})_{\sR_{23}}
$$
is a morphism of Hopf algebras.
\end{lem}

Corollary \ref{lem-Rmatrix-twist} below is essentially a reformulation of the above Lemma.

\begin{cor} \label{lem-Rmatrix-twist}
Let $(H, R)$ be a quasitriangular Hopf algebra and let $A_1, A_2$ be Hopf algebras equipped with morphisms of Hopf algebras $\varphi_j: H \to A_j$, $j = 1,2$, and let $\varphi = (\varphi_1 \otimes \varphi_2): H^{\otimes 2} \to A_1 \otimes A_2$. Then $J = (\varphi \otimes \varphi)(R_{23})$ is a twisting element of $A_1 \otimes A_2$, and 
$$
\varphi \circ \Delta_\sH: H \to (A_1 \otimes A_2)_\sJ
$$ 
is a morphism of Hopf algebras.
\end{cor}

\subsection{The quasitriangular Hopf algebras $(H^{(m)}_\sR, R^{(m)})$} \label{subsec-A^(n)_R}

Let $(H, R)$ be a quasitriangular Hopf algebra. By applying Corollary \ref{lem-Rmatrix-twist} inductively, one obtains for any $m \geq 1$ a Hopf algebra $H_\sR^{(m)}$, such that $H_\sR^{(1)} = H$, $H_\sR^{(2)} = (H^{\otimes 2})_{\sR_{23}}$, and for $m \geq 3$, $H_\sR^{(m)}$ is the twist of $H^{(m-1)}_\sR \otimes H$ by 
\begin{equation} \label{eqn-T^(n))}
(\Delta_\sH^{(m-1)} \otimes I_\sH \otimes \Delta_\sH^{(m-1)} \otimes I_\sH)(R_{23}).
\end{equation}
It follows that $H_\sR^{(m)}$ is also the twist of $H^{\otimes m}$ by $\Twi^m(R) \in (H^{\otimes m}) \otimes (H^{\otimes m})$, where $\Twi^1(R) = 1_\sH \otimes 1_\sH$, and for $m \geq 2$, 
\begin{equation} \label{J^n-inductive}
\Twi^m(R) = \Twi^{m-1}(R)(\Delta_\sH^{(m-1)} \otimes I_\sH \otimes \Delta_\sH^{(m-1)} \otimes I_\sH)(R_{23}) \in (H^{\otimes m}) \otimes (H^{\otimes m}), 
\end{equation}
where one views $\Twi^{m-1}(R)$ as an element of $(H^{\otimes m}) \otimes (H^{\otimes m})$ via the embedding 
$$
(H^{\otimes m-1}) \otimes  (H^{\otimes m-1}) \hookrightarrow (H^{\otimes m}) \otimes (H^{\otimes m}), \hs a_1 \otimes a_1 \mapsto a_1 \otimes 1_\sH \otimes a_2 \otimes 1_\sH, \hs a_1, a_2 \in H^{\otimes m-1}.
$$ 
Unravelling \eqref{J^n-inductive} using \eqref{hexagon},  one gets
\begin{equation} \label{J^n-closed}
\Twi^m(R) = \prod_{k=2}^m \:\:  \prod_{l=k-1}^1 R_{k \: m+l},
\end{equation}
where the outer product is taken in the increasing order of the indices, and the inner product in the decreasing order of the indices. In particular, $H_\sR^{(m)}$ has quasitriangular $R$-matrix
\begin{equation} \label{eq-R(n)}
R^{(m)} = \left( \prod_{k=m}^2 \:\:  \prod_{l=1}^{k-1} R_{l \: m+k}^{-1} \right) \left( \prod_{k=1}^m R_{k \: m+k} \right) \left(  \prod_{k=2}^m \:\:  \prod_{l=k-1}^1 R_{k \: m+l}\right), 
\end{equation}
and to sum up, one has the following

\begin{lem-defn}
Let $(H, R)$ be a quasitriangular Hopf algebra. For any $m \geq 1$, $H_\sR^{(m)}$ is a Hopf algebra with quasitriangular $R$-matrix $R^{(m)} \in H^{\otimes m} \otimes H^{\otimes m}$ given in \eqref{eq-R(n)}, and the $m$-fold comultiplication
$$
\Delta_\sH^{(m)}:H \to H^{(m)}_\sR
$$
is a morphism of Hopf algebras. We say that the quasitriangular Hopf algebra $(H_\sR^{(m)}, R^{(m)})$ is the {\it twisted $m$-fold tensor product} of $(H, R)$. 
\end{lem-defn}

Let now $H$ be a quantization of a quasitriangular Lie bialgebra $(\g, r)$ and $R \in H \hat{\otimes} H$ a quantization of $r$. 

\begin{pro}
For any $m \geq 1$, $H_\sR^{(m)}$ is a quantization of $(\g^m, \delta_{r^{(m)}})$ and $R^{(m)}$ is quantization of $r^{(m)}$. 
\end{pro}
\begin{proof}
Write $r = \sum_i x_i \otimes y_i$ with $x_i, y_i \in \g$. Since $r = \frac{1}{\hslash} (R - 1_\sH \otimes \sH) \; {\rm mod} \: \hslash$, the order $1$ term of $\Twi^m(R)$ is 
\begin{align*}
j^{(m)}& = \frac{1}{\hslash}(\Twi^m(R) - 1_\sH \otimes \cdots \otimes 1_\sH) \; \; {\rm mod} \: \hslash \in \g^m \otimes \g^m   \\
   & = \sum_{1 \leq l < k \leq m} \sum_i (x_i)_k \otimes (y_i)_l,
\end{align*}
and the order $1$ term of $R^{(m)}$ is 
$$
(r, \ldots, r) - (j^{(m)})_{21} + j^{(m)} = (r, \ldots, r) - \Mix^m(r) = r^{(m)},
$$
proving the Proposition. 
\end{proof}

\subsection{Module algebras over $(H^{(m)}_\sR, R^{(m)})$} \label{subsec-loc-fac-alg}

Let $(H, R)$ be a quasitriangular Hopf algebra and for $m \geq 1$, let $A_1, \ldots, A_m$  be left $H$-module algebras, so that the tensor product algebra $A = A_1 \otimes \cdots \otimes A_m$ is a left $H^{\otimes m}$-module algebra. By applying Lemma \ref{pro-twist-module}, one obtains the $H_\sR^{(m)}$-module algebra
\begin{equation} \label{eq-B^(R)}
\hat{A} = (A)_{{\scriptscriptstyle \Twi^m(R)}},
\end{equation}
the twist of $A$ by $\Twi^m(R)$ in the sense of Definition \ref{defn-twist-module}. In the following Proposition \ref{thm-B^(n)_R},  we describe the multiplication in $\hat{A}$, and show that $\hat{A}$ is a {\it locally factored algebra} in the sense of Etingof and Kazhdan, whose definition we now recall.

\begin{defn} \cite[Section 1.1]{quant-li-bialg-III}
A {\it locally factored algebra} is an algebra $A$ with a collection $A_1, \ldots, A_m$ of subalgebras such that for any permutation $\sigma \in S_m$ of $m$ letters, the multiplication map $A_{\sigma(1)} \otimes \cdots \otimes A_{\sigma(m)} \to A$ is a bijection, and such that for all $1 \leq i,j \leq m$, the image of $A_i \otimes A_j$ under the multiplication map is a subalgebra of $A$. 
\end{defn}

For any $1 \leq j \leq m$, we view $A_j$ as a subalgebra of $A$ by identifying it with the $j$'th tensor factor of $A$. 

\begin{pro} \label{thm-B^(n)_R} 
For $1 \leq i,j \leq m$, $f \in A_i$ and $g \in A_j$, one has
\begin{equation} \label{eq-B^(n)_R}
\mu_{\scriptscriptstyle \hat{A}}(f \otimes g)  = 
\left\{ \begin{array}{ll}
\mu_\sA(f \otimes g), & \text{if} \; i \leq j,   \\
R_{ij} \cdot \mu_\sA(g \otimes f), & \text{if} \; i > j.
\end{array} \right. 
\end{equation}
In particular, $\hat{A}$ is a locally factored algebra. 
\end{pro}

We first treat the case $m=2$ in the following Lemma \ref{lem-module-alg}. 

\begin{lem} \label{lem-module-alg}
Proposition \ref{thm-B^(n)_R} holds for $m = 2$. 
\end{lem}
\begin{proof}
Let $i = j= 1$. One has using \eqref{eqn-eps-I-R}, 
\begin{align*}
\mu_{\scriptscriptstyle \hat{A}}(f \otimes g) & =  \mu_\sA(R_{23} \cdot (f \otimes g)) = \mu_\sA(f \otimes \left( (\varepsilon_\sA \otimes I_\sA)(R) \cdot g \right) ) =  \mu_\sA(f \otimes g),
\end{align*}
and a similar calculation shows that $\mu_{\scriptscriptstyle \hat{A}}(f \otimes g) =  \mu_\sA(f \otimes g)$, if $i = j =2$, or $i = 1$, and $j = 2$. If $i = 2$ and $j = 1$, then 
\begin{align*}
\mu_{\scriptscriptstyle \hat{A}}(f \otimes g) & =  \mu_\sA(R_{23} \cdot (f \otimes g)) = R_{21} \cdot \mu_\sA(g \otimes f).
\end{align*}
\end{proof}

\noindent
{\it Proof of Proposition \ref{thm-B^(n)_R}}
We proceed by induction, the case $m=1$ being trivial. Thus let $m \geq 2$ and let $\tilde{A} = (A_1 \otimes \cdots \otimes A_{m-1})_{\scriptscriptstyle \Twi^{m-1}(R)}$. Viewing $\tilde{A}$ as an $H$-module algebra via the Hopf algebra morphism $\Delta_\sH^{(m-1)}$, by \eqref{J^n-inductive} one has 
$$
\hat{A} = (\tilde{A} \otimes A_n)_{{\scriptscriptstyle R_{23}}}.
$$
Hence if $i \leq j$ or $j < i < n$, Proposition \ref{thm-B^(n)_R}  follows from induction and Lemma \ref{lem-module-alg}. If $j < i = n$, by Lemma \ref{lem-module-alg}, one has 
\begin{align*}
\mu_{\scriptscriptstyle \hat{A}}(f \otimes g) & = (\Delta_\sH^{(n-1)} \otimes I_\sH)(R_{21}) \cdot \mu_\sA(g \otimes f)   \\
   & = (R_{n \: n-1}R_{n \: n-2} \cdots R_{n1}) \cdot \mu_\sA(g \otimes f)   \\
   & = R_{n j} \cdot \mu_\sA(g \otimes f),
\end{align*}
which concludes the proof. 
\qed

\begin{rem}
The algebra $\hat{A}$ is an $H$-module algebra via $\Delta_\sH^{(m)}$, and Proposition \ref{thm-B^(n)_R} has the following categorical interpretation. Let $\C$ be a braided monoidal category with braiding $\beta$, and let $A_1, A_2$ be two associative algebra objects in $\C$. Then $A_1 \otimes A_2$ is also an associative algebra object, with multiplication 
$$
A_1 \otimes A_2 \otimes A_1 \otimes A_2 \stackrel{1_{\sA_1} \otimes \beta_{\sA_2, \sA_1} \otimes 1_{\sA_2}}{\longrightarrow} A_1 \otimes A_1 \otimes A_2 \otimes A_2 \stackrel{\mu_{\sA_1} \otimes \mu_{\sA_2}}{\longrightarrow} A_1 \otimes A_2. 
$$
Applied to the category $\C$ of representations of the quasitriangular Hopf algebra $(H, R)$ whose braiding is 
$$
\beta_{\scriptscriptstyle A_1, A_2} =\tau_{(12)} \circ R \mid_{\scriptscriptstyle A_1, \: A_2}, \hs A_i \in {\rm Ob}(\C), 
$$
one precisely obtains the $H$-module algebra $(A_1 \otimes A_2)_{\scriptscriptstyle \Twi^2(R)}$. 
\hfill $\diamond$
\end{rem}

\begin{rem}
It is a consequence of \cite[Corollary 1.4]{quant-li-bialg-III} that formula \eqref{eq-B^(n)_R} defines an associative algebra which is locally factored with components $A_1, \ldots, A_m$. Proposition \ref{thm-B^(n)_R} gives an interpretation of this Corollary in terms of module algebras over twisted tensor products of quasitriangular Hopf algebras. 
\hfill $\diamond$
\end{rem}

\subsection{Quantization of mixed product Poisson structures}

Let $(U_\hslash(\g), R)$ be a quasitriangular QUE algebra quantizing a quasitriangular Lie bialgebra $(\g, r)$,  let $A$ be a DQ algebra quantizing a Poisson algebra $(A_0, \{ \:,\:\})$, and assume that $A$ is an $U_\hslash(\g)$-module algebra. Then $\rho: \g \to \Der(A_0)$, defined in \eqref{eq-equiv-quant}, is a Poisson action of $(\g, r)$ on $(A_0, \{ \:,\:\})$ and for $m \geq 1$, let 
$$
\{\:, \:\}^{(m)}
$$
be the mixed product Poisson structure on $A_0^{\otimes m}$ defined in $\S$\ref{subsec-poiss-act}. 

\begin{thm} \label{pro-quant-mix}
The algebra $(A^{\otimes m})_{\scriptscriptstyle \Twi^m(R)}$ is a quantization of $(A_0^{\otimes m}, \{\:, \:\}^{(m)})$.
\end{thm}
\begin{proof}
For $1 \leq k \leq m$, let $(A)_k$ be the image of the embedding of $A$ into $ A^{\otimes m}$ as the $k$'th component, and use a similar notation for $A_0^{\otimes m}$. As $(A)_k$ is a subalgebra of $(A^{\otimes m})_{\scriptscriptstyle \Twi^m(R)}$, it follows that $(A_0)_k$ is a Poisson subalgebra of $(A_0^{\otimes m}, \{\:, \:\}^{(m)})$. Let $1 \leq k < l \leq m$ and  $\tilde{f}, \tilde{g} \in A$. By \eqref{eq-B^(n)_R} one has 
\begin{align*}
\frac{1}{\hslash} (\mu_{\scriptscriptstyle (A^{\otimes m})_{\Twi^m(R)}}(\tilde{f}_k \otimes \tilde{g}_l) - &\mu_{\scriptscriptstyle (A^{\otimes m})_{\Twi^m(R)}}(\tilde{g}_l \otimes \tilde{f}_k))  \; \; {\rm mod} \: \hslash \\
  & = \frac{1}{\hslash}(1 \otimes \cdots \otimes 1 - R_{lk}) \mu_{\scriptscriptstyle A^{\otimes m}}(\tilde{f}_k \otimes \tilde{g}_l)  \; \; {\rm mod} \: \hslash   \\
  & = - \rho^{(m)}(r_{lk})(f \otimes g)_{kl}  \\
  & = - \rho^{(m)}(\Mix_{k,l}^m(r))(f_k \wedge g_l)   \\
  & = \{f_k, g_l\}^{(m)}, 
\end{align*}
where $f = \tilde{f} + \hslash A$,  $g = \tilde{g} + \hslash A$, which proves the Theorem. 
\end{proof}


\section{Twisted products of quantum principal affine spaces} \label{sec-quantum-NbackslashG}

Applying the theory developed in $\S$\ref{sec-twisted}, we construct a quantization of the Poisson bracket $\{ \:,\:\}^{(m)}$ on $\IC[N \backslash G]^{\otimes m}$ for any $m \geq 1$. Throughout this section, $\g$ is a complex semisimple Lie algebra as in $\S$\ref{subsec-std-PL} and $G$ is the connected and simply connected group integrating $\g$.

\subsection{The Drinfeld-Jimbo quantum group} \label{exa-drin-jimbo}

The celebrated {\it Drinfeld-Jimbo quantum group} is a quasitriangular QUE algebra quantizing the standard complex semisimple Lie bialgebra $(\g, r_{st})$, whose definition we now recall for the convenience of the reader. Our reference is \cite{drin:quantum.groups}. Let $q = e^\hslash \in \ch$ and recall the {\it $q$-integers}, {\it $q$-factorials}, and {\it $q$-binomial coefficients}, 
\begin{align*}
[n]_q & = \frac{q^n - q^{-n}}{q - q^{-1}} = \sum_{i=-n+1}^{n-1} q^i, \hs n \geq 1,   \\
[n]_q! & = \prod_{i=1}^n [i]_q, \hs \hs \left[\begin{array}{c}n \\i\end{array}\right]_q = \frac{[n]_q!}{[n-i]_q! [i]_q!}.
\end{align*}
For $1 \leq i \leq k$,  let $h_{\alpha_i} = [e_{\alpha_i}, e_{-\alpha_i}]$, and for $1 \leq i,j \leq k$, let $a_{ij} = \alpha_j(h_{\alpha_i})$, $d_i = \la \alpha_i, \alpha_i \ra$. Define $U_\hslash(\g)$ to be the deformation Hopf algebra generated by $h_{\alpha_i}$, $e_{\alpha_i}$, $e_{-\alpha_i}$, $1 \leq i \leq k$, and relations 
\begin{align*}
[h_{\alpha_i}, h_{\alpha_j}]  = 0, \hs  [h_{\alpha_j}, e_{\alpha_i}]  = a_{ij}e_{\alpha_i}, \hs [h_{\alpha_j}, e_{-\alpha_i}] & = -a_{ij}e_{-\alpha_i} \\
[e_{\alpha_i}, e_{-\alpha_j}]  = \delta_{ij} \frac{q_i^{h_{\alpha_i}} - q_i^{-h_{\alpha_i}}}{q_i - q_i^{-1}} \hs \hs &  \\
\sum_{l=0}^{1-a_{ij}} (-1)^l \left[\begin{array}{c} 1-a_{ij} \\ l \end{array}\right]_{q_i} e_{\pm \alpha_i}^{1-a_{ij} -l} e_{\pm \alpha_j} e_{\pm \alpha_i}^l & = 0,
\end{align*}
where $q_i = e^{\hslash d_i/2}$, $1 \leq i \leq k$. The Hopf algebra structures are given by 
\begin{alignat*}{6}
\Delta(e_{\alpha_i}) & = e_{\alpha_i} \otimes q_i^{h_{\alpha_i}/2} + q_i^{-h_{\alpha_i}/2} \otimes e_{\alpha_i},  & S(e_{\alpha_i}) & = - q_ie_{\alpha_i}, & \varepsilon(e_{\alpha_i}) & = 0,  \\
\Delta(e_{-\alpha_i}) & = e_{-\alpha_i} \otimes q_i^{h_{\alpha_i}/2} + q^{-d_i h_{\alpha_i}/2} \otimes e_{-\alpha_i}, & \hs S(e_{-\alpha_i}) & = -q_i^{-1}e_{-\alpha_i}, &\hs \varepsilon(e_{\alpha_i}) & = 0,   \\
\Delta(h_{\alpha_i}) & = h_{\alpha_i} \otimes 1 + 1 \otimes h_{\alpha_i}, & S(h_{\alpha_i}) & = -h_{\alpha_i}, & \varepsilon(h_{\alpha_i}) & = 0.
\end{alignat*}

\subsection{The quantum principal affine space} 

We briefly recall the construction of the quantum principal affine space of $G$, see e.g \cite{joseph, koro-soibelman}.

A {\it finite dimensional $U_\hslash(\g)$-module} is a topologically free $U_\hslash(\g)$-module $V$ such that $V_0 = V/\hslash V$ is a finite dimensional vector space. Since $V_0$ is naturally a $U(\g)$-module, one has a functor $F: \Rep(U_\hslash(\g)) \to \Rep(\g)$, $F(V) = V/\hslash V$, from the category of finite dimensional $U_\hslash(\g)$-modules to the category of finite dimensional representations of $\g$, and $F$ defines a bijection \cite[Theorem 5.1.2]{koro-soibelman} between isomorphism classes in $\Rep(U_\hslash(\g))$ and $\Rep(\g)$. For $V \in \Rep(U_\hslash(\g))$, $v \in V$ and $\xi \in V^*$, one has the {\it matrix coefficient} $c^\sV_{\xi, v} \in U_\hslash(\g)^*$ defined by 
$$
c^\sV_{\xi, v}(x) = \la \xi, \: S(x) \cdot v \ra = \la x \cdot \xi, \: v \ra, \hs \hs x \in U_\hslash(\g). 
$$
The {\it quantized ring of functions on $G$} is the subspace 
\begin{equation} \label{eq-IC_hG}
\IC_\hslash[G] = \{ c^\sV_{\xi, v}: V \in \Rep(U_\hslash(\g)), \; v \in V, \; \xi \in V^*\} \subset U_\hslash(\g)^*,
\end{equation}
with multiplication given by 
$$
c^\sV_{\xi, v} \: c^\sW_{\rho, w} = c^{\scriptscriptstyle V \otimes W}_{\xi \otimes \rho, \:\: v \otimes w},  \hs \hs c^\sV_{\xi, v},  c^\sW_{\rho, w} \in \IC_\hslash[G],   
$$
and $\IC_\hslash[G]$ is a quantization of the Poisson algebra $(\IC[G], \{ \:, \:\}^{(1)}_{r_{st}})$. One has the left and right actions of $U_\hslash(\g)$ on $\IC_\hslash[G]$ given by 
\begin{align*}
(y \cdot c^\sV_{\xi, v})(x) & = c^\sV_{\xi, v}(xy) = c^\sV_{y \cdot \xi, \: v}(x),   \\
(c^\sV_{\xi, v} \cdot y)(x) & = c^\sV_{\xi, v}(yx) = c^\sV_{\xi, \: S(y) \cdot v}(x), \hs \hs c^\sV_{\xi, v} \in \IC_\hslash[G], \; x, y \in U_\hslash(\g),
\end{align*}
and $\IC_\hslash[G]$ is an $U_\hslash(\g)$-bimodule algebra. 

For $\varpi \in \P_+$, let $V_\hslash(\varpi) \in \Rep(U_\hslash(\g))$ be such that $F(V_\hslash(\varpi)) = V(\varpi)$ and fix a highest weight vector $v \in V_\hslash(\varpi)$, i.e a non-zero vector annihilated by $e_{\alpha_i}$, $i = 1, \ldots, k$. Let 
\begin{align*}
\IC_\hslash[G]^\varpi & = \{c^{\scriptscriptstyle V_\hslash(\varpi)}_{\xi, v}, \: \xi \in V_\hslash(-w_0(\varpi)) \}, \\
\IC_\hslash[N \backslash G] & = \bigoplus_{\varpi \in \P_+} \IC_\hslash[G]^\varpi. 
\end{align*}
Then $\IC_\hslash[N \backslash G]$ is a quantization of $(\IC[N \backslash G], \{ \:, \:\}^{(1)})$, and is called the {\it quantum principal affine space} of $G$ \cite[Chapter 9]{joseph}.

\subsection{Quantization of products of principal affine spaces}

Let $U_\hslash(\t)$ be the subalgebra of $U_\hslash(\g)$ generated by $h_{\alpha_i}$, $i = 1, \ldots, k$. It is cocommutative and isomorphic to the commutative algebra $\IC[\t][[\hslash]]$. The $U_\hslash(\g)$-bimodule algebra structure on $\IC_\hslash[G]$ induces a left action of the tensor product QUE algebra 
$$
\widetilde{U_\hslash(\g)} = U_\hslash(\g) \otimes U_\hslash(\t)
$$
on $\IC_\hslash[N \backslash G]$ by 
$$
(y \otimes x) \cdot c^{\scriptscriptstyle V_\hslash(\varpi)}_{\xi, v} = \la \varpi, x \ra c^{\scriptscriptstyle V_\hslash(\varpi)}_{y \cdot \xi, \: v}, \hs \hs  c^{\scriptscriptstyle V_\hslash(\varpi)}_{\xi, v} \in \IC_\hslash[G]^\varpi, \;\; x \in U_\hslash(\t), \; y \in U_\hslash(\g),
$$
where we denote the $\hslash$-linear extension of $\varpi \in \t^*$ by the same letter. Hence $\IC_\hslash[N \backslash G]$ is a left $\widetilde{U_\hslash(\g)}$-module algebra which quantizes the Poisson action $\rho$ defined in \eqref{eq-tilde-rho} of $(\tilde{\g}, \tilde{r}_{st})$ on $(\IC[N \backslash G], \{ \:, \:\}^{(1)})$ in the sense of Lemma \ref{lem-equiv-quant}. Let $R \in U_\hslash(\g) \hat{\otimes} U_\hslash(\g)$ be a quasitriangular $R$-matrix quantizing $r_{st} \in \g \otimes \g$. As
$$
R_0 = e^{\hslash r_0} \in U_\hslash(\t) \hat{\otimes} U_\hslash(\t)
$$
is a quasitriangular $R$-matrix quantizing $r_0 \in \t \otimes \t$, $(\widetilde{U_\hslash(\g)}, \tilde{R})$ is a quasitriangular QUE algebra quantizing $(\tilde{\g}, \tilde{r}_{st})$, where 
$$
\tilde{R} = \tau_{(23)} (R \otimes R_0^{-1}) \in \widetilde{U_\hslash(\g)} \hat{\otimes} \widetilde{U_\hslash(\g)}, 
$$
and $\tau_{(23)}$ is the permutation of the second and third tensor factors. Let 
$$
\IC_\hslash[(N \backslash G)^m] = (\IC_\hslash[N \backslash G]^{\otimes m}){\scriptscriptstyle \Twi^m(\tilde{R})}. 
$$
Theorem \ref{main-thm2-quant-affine} below is now a consequence of Theorem \ref{pro-quant-mix}.

\begin{thm} \label{main-thm2-quant-affine}
For any $m \geq 1$, $\IC_\hslash[(N \backslash G)^m]$ is a quantization of $(\IC[N \backslash G]^{\otimes m}, \{ \:, \:\}^{(m)})$. 
\end{thm}

\begin{rem}
Using formula \eqref{eq-B^(n)_R}, one gets an explicit formula for the multiplication in $\IC_\hslash[(N \backslash G)^m] $ depending only on $R$ (and the multiplication in $\IC_\hslash[N \backslash G]$). 
\hfill $\diamond$
\end{rem}


\section{Strongly coisotropic subalgebras and graded algebras} \label{sec-strong-coiso-hopf}

We introduce an analogue of Definition \ref{defn-strong-co} in the context of Hopf algebras, and show how to construct graded algebras out of actions by strongly coisotropic subalgebras of Hopf algebras. As an application, we obtain another construction of $\IC_\hslash[(N \backslash G)^m] $.

\subsection{Strongly coisotropic subalgebras}

If $U$ is an associative algebra, let $[U,U]$ be the two-sided ideal generated by the commutators of $U$, and let $\Ch(U)$ be the set of characters of $U$, that is algebra morphisms from $U$ to the ground field $\IK$. 

\begin{defn}
Let $H$ be a Hopf algebra and $U \subset H$ a subalgebra. We say that $U$ is a {\it right (resp. left) strongly coisotropic subalgebra} of $H$ if 
$$
\Delta_\sH(U) \subset U \otimes U + H \otimes [U,U] \hs \text{(resp. } \Delta_\sH(U) \subset U \otimes U + [U,U] \otimes H \text{)}. 
$$
\end{defn}

Let $H$ be a Hopf algebra and $A$ a right $H$-module algebra. If $U \subset H$ is a subalgebra and $ \zeta \in \Ch(U)$, let 
$$
A^\zeta = \{ f \in A : f \cdot u = \zeta(u)f, \; u \in U\}, 
$$
and let $A^\sU$ be the $U$-semi-invariant elements of $A$, that is 
\begin{align*}
A^\sU & = \bigoplus_{\zeta \in \P_\sA(U)} A^\zeta, \hs \text{where}   \\
\P_\sA(U) & = \{ \zeta \in \Ch(U) : A^\zeta \neq 0 \}. 
\end{align*}

\begin{pro} \label{thm-strong-coiso}
If $U \subset H$ is a right or left strongly coisotropic subalgebra, $\Ch(U)$ has a natural monoid structure such that 
$$
A^{\zeta_1} A^{\zeta_2} \subset A^{\zeta_1 \zeta_2}, \hs \zeta_1, \zeta_2 \in \P_\sA(U),
$$
i.e $A^\sU$ is a subalgebra of $A$ graded by $\P_\sA(U)$. 
\end{pro}
\begin{proof}
We assume that $U$ is right strongly coisotropic, as the left case is similar. Since $H \otimes [U,U]$ is a two-sided ideal in $U \otimes U + H \otimes [U,U]$, one has a natural identification 
$$
(U \otimes U + H \otimes [U,U])/H \otimes [U,U] \cong U \otimes U, \hs u_1 \otimes u_2 + H \otimes [U,U] \mapsto u_1 \otimes u_2.
$$
Let $\tilde{\Delta}_\sU: U \to U \otimes U$ be the composition of $\Delta_\sH \!\!\mid_\sU$ with this identification. By construction $\tilde{\Delta}_\sU$ is co-associative, thus its dual map is an associative product on $U^*$. Hence one obtains an associative product on $\Ch(U)$ defined as  
$$
(\varphi \psi)(u): = (\varphi \otimes \psi)(\tilde{\Delta}_\sU(u)), \hs u \in U, \; \varphi, \psi \in \Ch(U). 
$$
If $f \in A^{\zeta_1}$, $g \in A^{\zeta_2}$, with $\zeta_1, \zeta_2 \in \P_\sA(U)$, one has  
$$
(fg) \cdot u  = \mu_\sA((f \otimes g) \cdot \Delta_\sH(u)) = (\zeta_1 \zeta_2)(u) fg,
$$
i.e $A^{\zeta_1} A^{\zeta_2} \subset A^{\zeta_1 \zeta_2}$. 
\end{proof}

\subsection{Strongly coisotropic subalgebras in twisted tensor products of quasitriangular Hopf algebras}

In Proposition \ref{lem-hopf-strong} below, we derive an analogue of Lemma \ref{lem-strong_co-r} for strongly coisotropic subalgebras. 

Let $A_1, A_2$ be Hopf algebras with Hopf subalgebras $U_i \subset A_i$, let $(H, R)$ be a quasitriangular Hopf algebra with morphisms $\varphi_i: H \to A_i$, $i = 1,2$, and let $J \in (A_1 \otimes A_2)^{\otimes 2}$ be as in Corollary \ref{lem-Rmatrix-twist}. 

\begin{lem} \label{lem-induction-coiso}
Assume that $R$ satisfies 
\begin{equation} \label{eq-R-U_i-compatibility}
(\varphi_2 \otimes \varphi_1)(R) \in U_2 \otimes U_1 + A_2 \otimes [U_1, U_1].
\end{equation}
Then $U_1 \otimes U_2$ is a right strongly coisotropic subalgebra of $(A_1 \otimes A_2)_\sJ$. 
\end{lem}
\begin{proof}
The proof is straightforward as $R^{-1} = (S_\sH \otimes 1_\sH)(R)$ also satisfies \eqref{eq-R-U_i-compatibility}. 
\end{proof}

\begin{pro} \label{lem-hopf-strong}
Let $U \subset H$ be a Hopf subalgebra and suppose that 
\begin{equation} \label{R-U-compatibility}
R \in U \otimes U + H \otimes [U,U]. 
\end{equation}
Then for any $m \geq 1$, $U^{\otimes m}$ is a right strongly coisotropic subalgebra of $H_\sR^{(m)}$. 
\end{pro}
\begin{proof}
The Proposition is trivial for $m=1$, so assume it true for an $m \geq 1$. As $\Delta_\sH^{(m)}([U,U]) \!\subset [U^{\otimes m}, U^{\otimes m}]$, one applies Lemma \ref{lem-induction-coiso} with  $A_1 = H_\sR^{(m)}$,  $A_2 = H$, $\varphi_1 = \Delta_\sH^{(m)}$, $\varphi_2 = 1_\sH$ to conclude that $U^{\otimes m+1}$ is a right strongly coisotropic subalgebra of  $H_\sR^{(m+1)}$. 
\end{proof}

\begin{exa} \label{ex-U_hb-strong-co}
Denote by $H$ the  QUE algebra $U_\hslash(\g)$ defined in $\S$\ref{exa-drin-jimbo} and let $U$ be the Hopf subalgebra generated by $\{h_{\alpha_i}, e_{\alpha_i}\}_{i=1}^k$. The standard quasitriangular $R$-matrix on $H$ satisfies \eqref{R-U-compatibility} (see e.g \cite[Chapter 8]{chari-pressley}) and for $m \geq 1$, one obtains another construction of the DQ algebra $\IC_\hslash[(N \backslash G)^m] $. 

Indeed, writing $\H = \IC_\hslash[G]$, the map $\mu^{(m)}: \H^{\otimes m} \otimes \H^{\otimes m} \to \H^{\otimes m}$, 
$$
\mu^{(m)}(f_1 \otimes f_2) = \mu_{\scriptscriptstyle \H^{\otimes m}}(\Twi^m(R) \cdot (f_1 \otimes f_2) \cdot \Twi^m(R)), \;\; f_i \in \H^{\otimes m},
$$
where $\Twi^m(R)$ is defined in \eqref{J^n-closed}, is an associative product on $\H^{\otimes m}$ such that 
$$
\H^{(m)} = (\H^{\otimes m}, \mu^{(m)}, 1_{\scriptscriptstyle \H}^{\otimes m})
$$ 
is an $H^{(m)}_\sR$-bimodule algebra. One has $\P_{\scriptscriptstyle \H^{(m)}}(U^{\otimes m}) = (\P_+)^m$ and 
$$
(\H^{(m)})^{\scriptscriptstyle U^{\otimes m}} = \bigoplus_{\ulamb \in (\P_+)^m} \IC_\hslash[G]^{\ulamb},
$$
where for $\ulamb = (\lambda_1, \ldots, \lambda_m) \in (\P_+)^m$, 
$$
\IC_\hslash[G]^{\ulamb} = \IC_\hslash[G]^{\lambda_1} \otimes \cdots \otimes \IC_\hslash[G]^{\lambda_m}.
$$
By Proposition \ref{lem-hopf-strong}, $U^{\otimes m}$ is a right strongly coisotropic subalgebra of $H_\sR^{(m)}$, so $(\H^{(m)})^{\scriptscriptstyle U^{\otimes m}}$ is a $(\P_+)^m$-graded subalgebra of $\H^{(m)}$, easily seen to be be isomorphic to $\IC_\hslash[(N \backslash G)^m] $. 

\hfill $\diamond$
\end{exa}

\subsection{Relation with the quantum sections of Ciccoli, Fioresi, Gavarini} \label{subsec-3italians}

Let $\H$ be a Hopf algebra, let $\U$ be a coalgebra with a right action of $\H$, and let $p: \H \to \U$ be an $\H$-equivariant surjective morphism of coalgebras. Given a {\it prequantum section} \cite[Definition 3.5]{3italians}, i.e a non-zero element $d \in \H$ satisfying 
\begin{equation} \label{eq-prequant}
\Delta_{\scriptscriptstyle \H}(d) \in d \otimes d + \ker(p) \otimes \H, 
\end{equation}
consider the subspace $\H^d = \oplus_{n \geq 0} \H_n^d$ of $\H$, where 
$$
\H_n^d = \{l \in \H : \Delta_{\scriptscriptstyle \H}(l) \in d^n \otimes l + \ker(p) \otimes \H \}.
$$
Then $d$ is said to be a {\it quantum section} if $\H^d$ is a graded subalgebra of $\H$. 

Let $(G, \piG)$ be a Poisson Lie group and $Q$ a closed coisotropic subgroup of $(G, \piG)$. When $\H$ is the quantized ring of regular functions on $(G, \piG)$, $\U$ the quantized coalgebra of functions on $Q$ and $p$ the restriction map, Ciccoli, Fioresi, and Gavarini \cite{3italians} interpret $\H^d$ as being the quantization of the homogeneous coordinate ring of an ample line bundle on $Q \backslash G$. 

Assume now that $\H$ is the dual Hopf algebra of a Hopf algebra $H$. That is, for any finite dimensional left representation $V$ of $H$, one has the matrix coefficients 
$$
c_{\xi, v}^V: H \to \IK, \hs c_{\xi, v}^V(h) = \la h \cdot \xi, \: v \ra, \hs h \in H, 
$$
defined for any $v \in V$ and $\xi \in V^*$, and $\H$ is the Hopf algebra consisting of all the matrix coefficients of $H$. Let $U \subset H$ be a subalgebra with $\Delta_\sH(U) \subset U \otimes H$, assume that $\U = \{c_{\xi, v}^V \mid_U : \; c_{\xi, v}^V \in \H\} \subset U^*$, and that $p: \H \to \U$ is the restriction of the matrix coefficients to $U$. Thus $p$ is $\H$-equivariant, where $\H$ acts on $\U$ by $p(f) \cdot l = p(fl)$, for $l,f \in \H$. 

Now if in addition $U$ is left strongly coisotropic, any prequantum section $d \in \H$ is automatically a quantum section. Indeed, by \eqref{eq-prequant}, $p(d)$ is a character on $U$, and since $d \in \H^d_1$, one has $p(d) \in \P_{\scriptscriptstyle \H}(U)$. Equipping $\Ch(U)$ with the monoid multiplication defined in Proposition \ref{thm-strong-coiso}, it is clear that $p(d)^n = p(d^n)$ for all  $n \geq 1$, and under the right action of $H$ on $\H$ given by 
$$
c_{\xi, v}^V \cdot h = c_{\xi, \: S_\sH(h) \cdot v}^V, \hs c_{\xi, v}^V \in \H,  
$$
one has 
$$
\H_n^d = \H^{p(d)^n}. 
$$
By Proposition \ref{thm-strong-coiso}, $\H^U$ is a $\P_{\scriptscriptstyle \H}(U)$-graded subalgebra of $\H$, hence  
$$
\H^d = \bigoplus_{n \geq 0} \H^{p(d)^n} \subset \H^U
$$
is a graded subalgebra. Thus strong coisotropicity can  be thought of as a condition guaranteeing the quantizability of any homogeneous coordinate ring of an ample line bundle on a homogenous space of a Poisson Lie group. 

\begin{exa}
Returning to Example \ref{ex-U_hb-strong-co} above, $\H^{(m)}$ is the dual Hopf algebra of $H^{(m)}_\sR$, and for $\ulamb = (\lambda_1, \ldots, \lambda_m) \in (\P_+)^m$, 
$$
S_\hslash(\ulamb) = \bigoplus_{n \geq 0}  \IC_\hslash[G]^{n \lambda_1} \otimes \cdots \otimes \IC_\hslash[G]^{n \lambda_m} \subset \H^{(m)}
$$
is a graded algebra quantizing the graded Poisson algebra $(S(\ulamb), \{ \:, \: \}^{(m)})$ corresponding to the flat Poisson connection $\nabla^{\ulamb}$ on $L(\ulamb)$. 
\hfill $\diamond$
\end{exa}


\bibliographystyle{plain} 

\end{document}